\documentclass{amsart}
\usepackage{amsmath}
\usepackage{amssymb}
\usepackage{amsthm}
\usepackage{epsfig}
\usepackage{amsfonts}
\usepackage{amscd}
\usepackage{mathrsfs}
\usepackage{enumerate}
\usepackage{latexsym}
\usepackage{url}

\newtheorem{theorem}{Theorem}[section]
\newtheorem{lemma}[theorem]{Lemma}

\newtheorem{question}[theorem]{Question}

\newenvironment{proof of claim}{\noindent\textbf{Proof of the claim.}}{\hfill{$\square$}\newline}

\theoremstyle{definition}
\newtheorem{definition}[theorem]{Definition}
\newtheorem{notation}[theorem]{Notation}
\newtheorem{example}[theorem]{Example}

\theoremstyle{remark}
\newtheorem{remark}[theorem]{Remark}
\numberwithin{equation}{section}

\begin{document}

\title[A Gelfand--Naimark type theorem]{A Gelfand--Naimark type theorem}

\author[M. Farhadi and M.R. Koushesh]{M. Farhadi and M.R. Koushesh$^*$}

\address{\textbf{[First author]} Department of Mathematical Sciences, Isfahan University of Technology, Isfahan 84156--83111, Iran.}

\email{m.farhadi@math.iut.ac.ir}

\address{\textbf{[Second author]} Department of Mathematical Sciences, Isfahan University of Technology, Isfahan 84156--83111, Iran and School of Mathematics, Institute for Research in Fundamental Sciences (IPM), P.O. Box: 19395--5746, Tehran, Iran.}

\email{koushesh@cc.iut.ac.ir}

\thanks{$^*$Corresponding author}

\thanks{The research of the second author was in part supported by a grant from IPM (No. 94030418).}




\begin{abstract}
Let $X$ be a completely regular space. For a non-vanishing self-adjoint Banach subalgebra $H$ of $C_B(X)$ which has local units we construct the spectrum $\mathfrak{sp}(H)$ of $H$ as an open subspace of the Stone--\v{C}ech compactification of $X$ which contains $X$ as a dense subspace. The construction of $\mathfrak{sp}(H)$ is simple. This enables us to study certain properties of $\mathfrak{sp}(H)$, among them are various compactness and connectedness properties. In particular, we find necessary and sufficient conditions in terms of either $H$ or $X$ under which $\mathfrak{sp}(H)$ is connected, locally connected and pseudocompact, strongly zero-dimensional, basically disconnected, extremally disconnected, or an $F$-space.
\end{abstract}

\maketitle


\section{Introduction}

Throughout this article by a \textit{space} we mean a \textit{topological space}. We adopt the definitions of \cite{E}, in particular, completely regular spaces as well as compact spaces (and therefore locally compact spaces) are assumed to be Hausdorff. The field of scalars is assumed to be the complex field $\mathbb{C}$, though all results hold true (with the same proof) in the real setting.

Let $X$ be a space. We denote by $C(X)$ the set of all continuous scalar valued mappings on $X$ and we denote by $C_B(X)$ the set of all bounded elements of $C(X)$. The set $C_B(X)$ is a Banach algebra with pointwise addition and multiplication and supremum norm. For any $f$ in $C(X)$, the \textit{zero-set} of $f$ is defined to be $f^{-1}(0)$ and is denote by $\mathrm{z}(f)$, the \textit{cozero-set} of $f$ is defined to be $X\setminus\mathrm{z}(f)$ and is denote by $\mathrm{coz}(f)$, and the \textit{support} of $f$ is defined to be $\mathrm{cl}_X\mathrm{coz}(f)$ and is denoted by $\mathrm{supp}(f)$. The set of all zero-sets of $X$ is denote by $\mathrm{z}(X)$ and the set of all cozero-sets of $X$ is denote by $\mathrm{coz}(X)$. We denote by $C_0(X)$ the set of all $f$ in $C(X)$ which vanish at infinity (that is, $|f|^{-1}([\epsilon,\infty))$ is compact for any positive $\epsilon$). Also, we denote by $C_{00}(X)$ the set of all $f$ in $C(X)$ whose support is compact.

In \cite{Kou2} the second author has obtained a commutative Gelfand--Naimark type theorem which shows that for a locally separable metrizable space $X$ the set $C_s(X)$ of all continuous bounded complex valued mappings whose support is separable, where $C_s(X)$ is provided with the supremum norm, is a Banach algebra which is isometrically isomorphic to $C_0(Y)$ for some locally compact space $Y$. The space $Y$ (which is unique up to homeomorphism) has been constructed explicitly as a subspace of the Stone--\v{C}ech compactification of $X$. Furthermore, the space $Y$ is countably compact, and if $X$ is non-separable, is also non-normal. In addition $C_0(Y)=C_{00}(Y)$. The theorems in \cite{Kou2} have motivated a series of subsequent results. (See \cite{Kou3}, \cite{Kou4} and \cite{Kou6}. See also \cite{AG} for a relevant result.) This in particular includes our result in this paper in which, when $X$ is a completely regular space, for a non-vanishing self-adjoint Banach subalgebra of $C_B(X)$ which has local units we find a locally compact space $Y$ such that $H$ is isometrically isomorphic to $C_0(Y)$. This result also follows from the celebrated commutative Gelfand--Naimark theorem. Here, we construct $Y$ explicitly as an open subspace of the Stone--\v{C}ech compactification of $X$. The space $Y$ contains $X$ as a dense subspace. Furthermore, it is unique up to homeomorphism and therefore coincides with the spectrum of $H$. The simple construction of $Y$ enables us to study some of its properties, among them are various compactness and connectedness properties. In particular, we find necessary and sufficient conditions in terms of either $H$ or $X$ under which $\mathfrak{sp}(H)$ is connected, locally connected and pseudocompact, strongly zero-dimensional, basically disconnected, extremally disconnected, or an $F$-space.

Throughout this article we will make critical use of the theory of the Stone--\v{C}ech compactification. We review some of the basic properties of the Stone--\v{C}ech compactification in the following and refer the reader to the texts \cite{E} and \cite{GJ} for further possible reading.

\subsection*{The Stone--\v{C}ech compactification}

Let $X$ be a completely regular space. A \textit{compactification} of $X$ is a compact space which contains $X$ as a dense subspace. The \textit{Stone--\v{C}ech compactification} of $X$, denoted by $\beta X$, is the (unique) compactification of $X$ which is characterized among all compactifications of $X$ by the property that every mapping in $C_B(X)$ is extendable to a (unique) mapping in $C(\beta X)$. For a mapping $f$ in $C_B(X)$ we denote this continuous extension by $f_\beta$. We make use of the following properties.
\begin{itemize}
\item The space $X$ is locally compact if and only if $X$ is open in $\beta X$.
\item For any $X\subseteq‎ T\subseteq\beta ‎X$ we have $\beta‎ T=\beta ‎X$.
\item For any open subspace $V$ of $\beta X$ we have
\[\mathrm{cl}_{\beta X} V=\mathrm{cl}_{\beta X} (V\cap X).\]
\item The closure in $\beta X$ of every open and closed subspace of $X$ is open and closed in $\beta X$.
\item For any two open and closed subspaces $U$ and $V$ of $X$ we have
\[\mathrm{cl}_{\beta X}(U\cap V)=\mathrm{cl}_{\beta X}U\cap\mathrm{cl}_{\beta X}V.\]
In particular, disjoint open and closed subspaces of $X$ have disjoint closures in $\beta X$.
\end{itemize}

\section{The Representation Theorem}

In this section we prove our representation theorem. This, for a completely regular space $X$, provides a way to represent a non-vanishing self-adjoint closed subalgebra of $C_B(X)$ which has local units as $C_0(Y)$ for some (unique) locally compact space $Y$.

We now proceed with a few definitions and lemmas.

For a completely regular space $X$ and a subset $H$ of $C_B(X)$ the following subspace $\lambda_HX$ of $\beta X$ has been defined in \cite{Kou3} (originally in \cite{Kou1} and \cite{Kou5}) and plays a crucial role in our study.

\begin{definition}
Let $X$ be a completely regular space. For a subset $H$ of $C_B(X)$ let
\[\lambda_HX=\bigcup_{f\in H}\mathrm{int}_{\beta X}\mathrm{cl}_{\beta X}\mathrm{coz}(f).\]
\end{definition}

Observe that $\lambda_HX$ is open in (the compact space) $\beta X$ and is therefore locally compact.

Recall that for a space $X$, a subset $H$ of $C_B(X)$ is said to
\begin{itemize}
\item be \textit{self-adjoint} if $H$ contains the complex conjugate $\overline{f}$ of any element $f$ in $H$ (where $\overline{f}(x)=\overline{f(x)}$ for any $x$ in $X$).
\item be \textit{non-vanishing} if for any $x$ in $X$ there is some $f$ in $H$ such that $f(x)\neq 0$.
\item \textit{separate points of $X$} if for any distinct elements $x$ and $y$ in $X$ there is some $f$ in $H$ such that $f(x)\neq f(y)$.
\item \textit{have local units} if for any closed subspace $A$ in $X$ and any neighborhood $U$ of $A$ in $X$ contained in $\mathrm{supp}(h)$ for some $h$ in $H$, there is some $f$ in $H$ such that $f\mid_A=\mathbf{1}$ and $f\mid_{X\setminus U}=\mathbf{0}$.
\end{itemize}

\begin{lemma}\label{RSH}
Let $X$ be a completely regular space and let $H$ be a subset of $C_B(X)$. Let $f$ be in $H$ and let $f_\beta:\beta X\rightarrow\mathbb{C}$ be the continuous extension of $f$. Then
\[\mathrm{coz}(f_\beta)\subseteq\lambda_HX.\]
\end{lemma}

\begin{proof}
Note that $\mathrm{coz}(f_\beta)\subseteq\mathrm{cl}_{\beta X}\mathrm{coz}(f_\beta)$ and $\mathrm{cl}_{\beta X}\mathrm{coz}(f_\beta)=\mathrm{cl}_{\beta X}(X\cap\mathrm{coz}(f_\beta))$, as $\mathrm{coz}(f_\beta)$ is open in $\beta X$ (and $X$ is dense in $\beta X$). Also, $X\cap\mathrm{coz}(f_\beta)=\mathrm{coz}(f)$, as $f_\beta$ extends $f$. Thus $\mathrm{coz}(f_\beta)\subseteq\mathrm{int}_{\beta X}\mathrm{cl}_{\beta X}\mathrm{coz}(f)$. But $\mathrm{int}_{\beta X}\mathrm{cl}_{\beta X}\mathrm{coz}(f)\subseteq\lambda_HX$ by the definition of $\lambda_HX$.
\end{proof}

\begin{lemma}\label{TFRS}
Let $X$ be a completely regular space and let $H$ be a non-vanishing subset of $C_B(X)$. Then
\[X\subseteq\lambda_HX.\]
\end{lemma}

\begin{proof}
Let $x$ be in $X$. There is some $f$ in $H$ such that $f(x)\neq 0$. By Lemma \ref{RSH} we have $\mathrm{coz}(f_\beta)\subseteq\lambda_HX$. But then $x$ is in $\mathrm{coz}(f_\beta)$, as $f_\beta$ extends $f$.
\end{proof}

\begin{definition}
Let $X$ be a completely regular space and let $H$ be a non-vanishing subset of $C_B(X)$. For any $f$ in $H$ denote
\[f_H=f_\beta\mid_{\lambda_HX}.\]
\end{definition}

Note that $X\subseteq\lambda_HX$ by Lemma \ref{TFRS}, thus, $f_H$ extends $f$.

\begin{lemma}\label{UU}
Let $X$ be a completely regular space and let $H$ be a non-vanishing subset of $C_B(X)$. Then, $f_H$ is in $C_0(\lambda_HX)$ for any $f$ in $H$.
\end{lemma}

\begin{proof}
Let $f$ be in $H$. Let $\epsilon>0$. Then $|f_\beta|^{-1}([\epsilon,\infty))\subseteq\lambda_HX$, as $|f_\beta|^{-1}([\epsilon,\infty))\subseteq\mathrm{coz}(f_\beta)$ and $\mathrm{coz}(f_\beta)\subseteq\lambda_HX$ by Lemma \ref{RSH}. Thus
\[|f_H|^{-1}([\epsilon,\infty))=\lambda_HX\cap|f_\beta|^{-1}([\epsilon,\infty))=|f_\beta|^{-1}([\epsilon,\infty))\]
is closed in $\beta X$ and is therefore compact.
\end{proof}

\begin{lemma}\label{ID}
Let $X$ be a completely regular space and let $H$ be a self-adjoint subalgebra of $C_B(X)$. Let $A$ be a compact subspace of $\lambda_HX$. Then
\[A\subseteq\mathrm{cl}_{\beta X}\mathrm{coz}(f)\]
for some $f$ in $H$.
\end{lemma}

\begin{proof}
Suppose that $A$ is compact. Using the definition of $\lambda_HX$, we have
\[A\subseteq\mathrm{int}_{\beta X}\mathrm{cl}_{\beta X}\mathrm{coz}(h_1)\cup\cdots\cup\mathrm{int}_{\beta X}\mathrm{cl}_{\beta X}\mathrm{coz}(h_n)\]
for some $h_1,\ldots,h_n$ in $H$. In particular,
\begin{eqnarray*}
A&\subseteq&\mathrm{cl}_{\beta X}\mathrm{coz}(h_1)\cup\cdots\cup\mathrm{cl}_{\beta X}\mathrm{coz}(h_n)\\
&=&\mathrm{cl}_{\beta X}(\mathrm{coz}(h_1)\cup\cdots\cup\mathrm{coz}(h_n))=\mathrm{cl}_{\beta X}\mathrm{coz}(h)
\end{eqnarray*}
where
\[h=|h_1|^2+\cdots+|h_n|^2=h_1\overline{h_1}+\cdots+h_n\overline{h_n}.\]
Then $h$ is in $H$, as $H$ is self-adjoint.
\end{proof}

The following lemma is a corollary of the Stone--Weierstrass theorem (see Theorem 8.1 and Corollary 8.3 of \cite{7}, Chapter V) and will be used in the proof of our next theorem.

\begin{lemma}\label{YYF}
Let $X$ be a locally compact space and let $H$ be a non-vanishing self-adjoint closed subalgebra of $C_0(X)$ which separates points of $X$. Then
\[H=C_0(X).\]
\end{lemma}

The following version of the Banach--Stone theorem (see Theorem 7.1 of \cite{Be}) will be used in the proof of our theorem. It states that for a locally compact $X$ the topology of $X$ determines and is determined by the algebraic structure of $C_0(X)$. (It turns out that for a locally compact space $X$ even the ring structure of $C_0(X)$ suffices to determine the topology $X$; see \cite{AAN}.)

\begin{lemma}\label{IJUG}
Let $S$ and $T$ be locally compact spaces. Then the normed algebras $C_0(S)$ and $C_0(T)$ are isometrically isomorphic if and only if the spaces $S$ and $T$ are homeomorphic.
\end{lemma}

We are now at a point to prove our main result.

\begin{theorem}\label{FS}
Let $X$ be a completely regular space. Let $H$ be a non-vanishing self-adjoint closed subalgebra of $C_B(X)$ which has local units. Then $H$ is isometrically isomorphic to $C_0(Y)$ for some unique $($up to homeomorphism$)$ locally compact space $Y$, namely, for $Y=\lambda_HX$. Moreover, the following are equivalent:
\begin{itemize}
\item[(1)] $H$ is unital.
\item[(2)] $H$ contains $\mathbf{1}$.
\item[(3)] $Y$ is compact.
\item[(4)] $Y=\beta X$.
\end{itemize}
\end{theorem}

\begin{proof}
Let
\[\psi:H\longrightarrow C_0(\lambda_HX)\]
be defined by $\psi(f)=f_H$ for any $f$ in $H$. The definition makes sense by Lemma \ref{UU}. We show that $\psi$ is an isometric isomorphism. It is clear that $\psi$ is injective and that $\psi$ is a homomorphism. (To see this, let $f$ and $g$ be in $H$. Then $(f+g)_H=f_H+g_H$, as $(f+g)_H$ and $f_H+g_H$ are continuous mappings on $\lambda_HX$ which are both identical to $f+g$ on $X$, which is contained and is therefore dense in $\lambda_HX$ by Lemma \ref{TFRS}. Similarly, $(fg)_H=f_Hg_H$. Also, if $f_H$ equals to $g_H$ then the restrictions of $f_H$ and $g_H$ to $X$ are also equal.) It therefore suffices to show that $\psi$ is surjective and preserves norms.

To show that $\psi$ is an isometry, let $f$ be in $H$. By continuity of $f$ we have
\[|f_H|(\lambda_HX)=|f_H|(\mathrm{cl}_{\lambda_HX}X)\subseteq\overline{|f_H|(X)}=\overline{|f|(X)}\subseteq [0,\|f\|],\]
where the bar denotes the closure in $\mathbb{R}$. This implies that $\|f_H\|\leq\|f\|$. That $\|f\|\leq\|f_H\|$ is clear, as $f_H$ extends $f$. This shows that $\psi$ is an isometry.

Next, we show that $\psi$ is surjective, that is
\[\psi(H)=C_0(\lambda_HX).\]
Note that $\psi(H)$ is a subalgebra of $C_0(\lambda_HX)$ which is also closed in $C_0(\lambda_HX)$, as it is an isometric image of the complete normed space $H$. By Lemma \ref{YYF} it therefore suffices to verify that $\psi(H)$ satisfies the following conditions:
\begin{enumerate}
\item $\psi(H)$ does not vanish on $\lambda_HX$;
\item $\psi(H)$ separates points of $\lambda_HX$;
\item $\psi(H)$ is a self-adjoint subalgebra of $C_0(\lambda_HX)$.
\end{enumerate}

To show $(2)$, suppose that $x$ and $y$ are distinct elements of $\lambda_HX$. Let $U$ be an open neighborhood of $x$ in $\beta X$ such that $\mathrm{cl}_{\beta X}U$ does not contain $y$ and
\begin{equation}\label{JGH}
\mathrm{cl}_{\beta X}U\subseteq\lambda_HX.
\end{equation}
Let $U'$ be an open neighborhood of $x$ in $\beta X$ such that $\mathrm{cl}_{\beta X}U'\subseteq U$. Let \[A=\mathrm{cl}_X(X\cap U').\]
Then $\mathrm{cl}_{\beta X}A\subseteq U$, and thus $X\cap U$ is a neighborhood of $A$ in $X$. By (\ref{JGH}), and since $\mathrm{cl}_{\beta X}U$ is compact, it follows from Lemma \ref{ID} that $\mathrm{cl}_{\beta X}U\subseteq\mathrm{cl}_{\beta X}\mathrm{coz}(h)$ for some $h$ in $H$. In particular $U\subseteq\mathrm{cl}_{\beta X}\mathrm{coz}(h)$, which, intersecting with $X$, gives $X\cap U\subseteq\mathrm{cl}_X\mathrm{coz}(h)=\mathrm{supp}(h)$. That is, $X\cap U$ is a neighborhood in $X$ of the closed subspace $A$ of $X$ which is contained in $\mathrm{supp}(h)$. By our assumption, there is some $f$ in $H$ such that
\[f\mid_A=\mathbf{1}\quad\text{and}\quad f\mid_{X\setminus U}=\mathbf{0}.\]
Note that $\mathrm{cl}_{\beta X}U'=\mathrm{cl}_{\beta X}(X\cap U')=\mathrm{cl}_{\beta X}A$. We have
\[f_H(x)=f_\beta(x)\in f_\beta(\mathrm{cl}_{\beta X}U')=f_\beta(\mathrm{cl}_{\beta X}A)\subseteq \overline{f_\beta(A)}=\overline{f(A)}\subseteq\{1\},\]
where the bar denotes the closure in $\mathbb{C}$ and $f_\beta:\beta X\rightarrow\mathbb{C}$ denotes the continuous extension of $f$. Observe that
\[\mathrm{cl}_{\beta X}(X\cap U)\cup\mathrm{cl}_{\beta X}(X\setminus U)=\beta X.\]
Therefore, by the choice of $U$, we have
\[f_H(y)=f_\beta(y)\in f_\beta(\beta X\setminus\mathrm{cl}_{\beta X}U)\subseteq f_\beta(\mathrm{cl}_{\beta X}(X\setminus U))\subseteq\overline{f_\beta(X\setminus U)}=\overline{f(X\setminus U)}\subseteq\{0\}.\]
This implies that $f_H(x)\neq f_H(y)$. Thus $\psi(H)$ separates points of $\lambda_HX$. This shows (2).

Note that the above argument also proves (1), as for any $x$ in $\lambda_HX$ one can choose some $y$ in $\lambda_HX$ different from $x$ (which is always possible provided that $X$ is not a singleton!) and find an element $f_H$ in $\psi(H)$ for some $f$ in $H$ which does not vanish at $x$.

To prove (3), observe that for any $f$ in $H$, since $H$ is self-adjoint, $\overline{f}$ is in $H$, and therefore $\overline{f}_H$ is in $\psi(H)$. But $\overline{f}_H$ and $\overline{f_H}$ are identical (as they are continuous mappings on $\lambda_HX$ which agree on the dense subspace $X$ of $\lambda_HX$).

This shows that $\psi$ is an isometric isomorphism, and therefore, $H$ and $C_0(Y)$ are isometrically isomorphic, where $Y=\lambda_HX$. The fact that $Y$ is unique (up to homeomorphism) is immediate and follows from Lemma \ref{IJUG}.

We now verify the final assertion of the theorem. Let $H$ be unital with the unit element $u$. By our assumption, since $H$ is non-vanishing, for every $x$ in $X$ there exists some $f_x$ in $H$ such that $f_x(x)\neq 0$. But $u(x)f_x(x)=f_x(x)$ which yields $u(x)=1$. That is $u=\mathbf{1}$ and thus $H$ contains $\mathbf{1}$. Clearly, if $H$ contains $\mathbf{1}$ then $Y=\beta X$, by the definition of $\lambda_HX$. Also, if $Y$ is compact, since it contains $X$, it contains its closure in $\beta X$. Therefore $Y=\beta X$. Finally, if $Y=\beta X$, then $H$ is unital, as it is isometrically isomorphic to $C_0(Y)$ and the latter is so, since $C_0(Y)=C_B(Y)$.
\end{proof}

\begin{remark}
The closed subalgebra $H$ of $C_B(X)$ in the statement of Theorem \ref{FS} can be thought of as being a $C^*$-algebra with the standard operation $*:H\rightarrow H$ of complex conjugation (which maps $f$ to $\overline{f}$ for every $f$ in $H$). In this case, as is pointed out in the proof of Theorem \ref{FS}, we have $\psi(*(f))=*(\psi(f))$ for every $f$ in $H$, and thus, the mapping $\psi$ is an isometric $*$-isomorphism. In particular, $H$ and $C_0(Y)$ are then isometrically $*$-isomorphic.
\end{remark}

In the following we give examples of spaces $X$ and non-vanishing self-adjoint closed subalgebra $H$ of $C_B(X)$ which have local units (thus satisfying the requirements in  Theorem \ref{FS}).

Let $\mathscr{P}$ be a topological property. Then
\begin{itemize}
\item $\mathscr{P}$ is \textit{closed hereditary} if any closed subspace of a space which has $\mathscr{P}$ also has $\mathscr{P}$.
\item $\mathscr{P}$ is \textit{preserved under countable closed unions} if any space which is a countable union of its closed subspaces each having $\mathscr{P}$ also has $\mathscr{P}$.
\end{itemize}
A space $X$ is called \textit{locally-$\mathscr{P}$} if every point of $X$ has a neighborhood in $X$ which has $\mathscr{P}$.

\begin{example}
Let $\mathscr{P}$ and $\mathscr{Q}$ be topological properties such that
\begin{itemize}
\item $\mathscr{P}$ and $\mathscr{Q}$ are closed hereditary.
\item A space with both $\mathscr{P}$ and $\mathscr{Q}$ is Lindel\"{o}f.
\item $\mathscr{P}$ is preserved under countable closed unions.
\item A space with $\mathscr{Q}$ having a dense subspace with $\mathscr{P}$ has $\mathscr{P}$.
\end{itemize}
Let $X$ be a normal locally-$\mathscr{P}$ space with $\mathscr{Q}$. Let
\[C^\mathscr{P}_0(X)=\{f\in C_B(X):\mbox{$|f|^{-1}([1/n,\infty))$ has $\mathscr{P}$ for each $n$}\}.\]
We check that $C^\mathscr{P}_0(X)$ is a non-vanishing self-adjoint closed subalgebra of $C_B(X)$ which has local units. The fact that $C^\mathscr{P}_0(X)$ is a non-vanishing closed subalgebra of $C_B(X)$ is proved in Theorem 3.3.9 of \cite{Kou6}. It is also clear that $C^\mathscr{P}_0(X)$ contains $\overline{f}$ if it contains $f$. That is,  $C^\mathscr{P}_0(X)$ is self-adjoint. To show that $C^\mathscr{P}_0(X)$ has local units let $A$ be a closed subspace of $X$ and $U$ be a neighborhood of $A$ in $X$ which is contained in $\mathrm{supp}(h)$ for some $h$ in $H$. Observe that
\[C^\mathscr{P}_0(X)=\{f\in C_B(X):\mathrm{supp}(f)\mbox{ has }\mathscr{P}\}\]
by Theorem 3.3.10 of \cite{Kou6}. Since $X$ is normal, there is some $f$ in $C_B(X)$ such that $f\mid_A=\mathbf{1}$ and $f\mid_{X\setminus U}=\mathbf{0}$. Then $f$ is in $C^\mathscr{P}_0(X)$, as $\mathrm{supp}(f)$ has $\mathscr{P}$, since $\mathrm{supp}(f)$ is contained in $\mathrm{supp}(h)$ as a closed subspace, the latter has $\mathscr{P}$ and $\mathscr{P}$ is closed hereditary.

Specific examples of topological properties $\mathscr{P}$ and $\mathscr{Q}$ which satisfy the above requirements are given in Example 3.3.13 of \cite{Kou6}. This includes the case when $\mathscr{P}$ is the Lindel\"{o}f property and $\mathscr{Q}$ is paracompactness (and particularly, the case when $\mathscr{P}$ is either the Lindel\"{o}f property, second countability or separability and $\mathscr{Q}$ is metrizability).
\end{example}

Let $X$ be a completely regular space. In Theorem \ref{FS} for a non-vanishing self-adjoint closed subalgebra $H$ of $C_B(X)$ which has local units we have shown that $H$ is isometrically isomorphic to $C_0(\lambda_HX)$. It also follows from the commutative Gelfand--Naimark theorem that $H$ is isometrically isomorphic to $C_0(Y)$ in which $Y$ is the spectrum (or the maximal ideal space) of $H$ which has the Gelfand (or the Zariski) topology. The uniqueness part of Theorem \ref{FS} now implies that $Y=\lambda_HX$. We state this fact formally as a theorem.

\begin{notation}
Let $X$ be a space. We denote the spectrum of a closed subalgebra $H$ of $C_B(X)$ (considered as a $C^*$-subalgebra of $C_B(X)$ under the standard operation of complex conjugation) by $\mathfrak{sp}(H)$.
\end{notation}

\begin{theorem}\label{JUFG}
Let $X$ be a completely regular space. Let $H$ be a non-vanishing self-adjoint closed subalgebra of $C_B(X)$ which has local units. Then
\[\mathfrak{sp}(H)=\lambda_HX.\]
\end{theorem}

\section{Compactness and connectedness properties of spectrum}

In this section we use the representation given in Theorem \ref{FS} to study various properties of the spectrum.

\begin{lemma}\label{FYS}
Let $X$ be a normal space. Let $H$ be a non-vanishing self-adjoint closed subalgebra of $C_B(X)$ which has local units. Let
\[\psi:H\longrightarrow C_0(\lambda_HX)\]
be defined by $\psi(f)=f_H$ for any $f$ in $H$. Then
\[\psi^{-1}(C_{00}(\lambda_HX))=\{f\in H:\mathrm{supp}(f)\subseteq\mathrm{int}_X\mathrm{supp}(h)\mbox{ for some }h\in H\}.\]
\end{lemma}

\begin{proof}
Let $f$ be in $\psi^{-1}(C_{00}(\lambda_HX))$. Then $f_H$ is in $C_{00}(\lambda_HX)$, that is, $\mathrm{supp}(f_H)$ is a compact subspace of $\lambda_HX$. Let $U$ be an open subspace of $\beta X$ such that
\[\mathrm{supp}(f_H)\subseteq U\subseteq\mathrm{cl}_{\beta X}U\subseteq\lambda_HX.\]
Since $\mathrm{cl}_{\beta X}U$ is compact, by Lemma \ref{ID} we have
\[\mathrm{cl}_{\beta X}U\subseteq\mathrm{cl}_{\beta X}\mathrm{coz}(h)\]
for some $h$ in $H$. Clearly, $\mathrm{coz}(f)\subseteq\mathrm{coz}(f_H)$, as $f_H$ extends $f$. In particular $\mathrm{coz}(f)\subseteq\mathrm{supp}(f_H)$, and thus since the latter is closed in $\beta X$ (as it is compact) $\mathrm{cl}_{\beta X}\mathrm{coz}(f)\subseteq\mathrm{supp}(f_H)$. Now, using the above relations, we have
\[\mathrm{cl}_{\beta X}\mathrm{coz}(f)\subseteq U\subseteq\mathrm{cl}_{\beta X}\mathrm{coz}(h),\]
which if we intersect each side with $X$ it yields
\[\mathrm{supp}(f)=\mathrm{cl}_X\mathrm{coz}(f)\subseteq X\cap U\subseteq\mathrm{cl}_X\mathrm{coz}(h)=\mathrm{supp}(h).\]
Therefore $\mathrm{supp}(f)\subseteq\mathrm{int}_X\mathrm{supp}(h)$.

For the converse, let $f$ be in $H$ such that $\mathrm{supp}(f)\subseteq\mathrm{int}_X\mathrm{supp}(h)$ for some $h$ in $H$. Since $X$ is normal, by the Urysohn lemma, there is a continuous mapping $g:X\rightarrow [0,1]$ such that
\[g\mid_{\mathrm{supp}(f)}=\mathbf{1}\quad\text{and}\quad g\mid_{X\setminus\mathrm{int}_X\mathrm{supp}(h)}=\mathbf{0}.\]
Let $g_\beta:\beta X\rightarrow [0,1]$ be the continuous extension of $g$. Note that
\[g_\beta^{-1}((\tfrac{1}{2},1])\subseteq\mathrm{cl}_{\beta X}g_\beta^{-1}((\tfrac{1}{2},1])
=\mathrm{cl}_{\beta X}(X\cap g_\beta^{-1}((\tfrac{1}{2},1]))
=\mathrm{cl}_{\beta X}g^{-1}((\tfrac{1}{2},1])\]
and
\[\mathrm{cl}_{\beta X}g^{-1}((\tfrac{1}{2},1])\subseteq\mathrm{cl}_{\beta X}\mathrm{supp}(h)=\mathrm{cl}_{\beta X}\mathrm{coz}(h),\]
using the definition of $g$. Therefore
\[g_\beta^{-1}((\tfrac{1}{2},1])\subseteq\mathrm{int}_{\beta X}\mathrm{cl}_{\beta X}\mathrm{coz}(h).\]
Also
\[\mathrm{cl}_{\beta X}\mathrm{coz}(f)\subseteq g_\beta^{-1}(1)\]
by the definition of $g$ and $\mathrm{int}_{\beta X}\mathrm{cl}_{\beta X}\mathrm{coz}(h)\subseteq\lambda_HX$ by the definition of $\lambda_HX$. Thus $\mathrm{cl}_{\beta X}\mathrm{coz}(f)\subseteq\lambda_HX$. It now follows that
\begin{eqnarray*}
\mathrm{supp}(f_H)&=&\mathrm{cl}_{\lambda_HX}\mathrm{coz}(f_H)\\&=&\mathrm{cl}_{\lambda_HX}(X\cap\mathrm{coz}(f_H))\\
&=&\mathrm{cl}_{\lambda_HX}(\mathrm{coz}(f))\\&=&\lambda_HX\cap\mathrm{cl}_{\beta X}\mathrm{coz}(f)=\mathrm{cl}_{\beta X}\mathrm{coz}(f)
\end{eqnarray*}
is compact. (Note that $X\subseteq\lambda_HX$ by Lemma \ref{TFRS}.) Therefore $f_H$ is in $C_{00}(\lambda_HX)$. That is, $f$ is in $\psi^{-1}(C_{00}(\lambda_HX))$.
\end{proof}

Let $X$ be a locally compact space. It is known that $C_0(X)=C_{00}(X)$ if and only if every $\sigma$-compact subspace of $X$ is contained in a compact subspace of $X$. (See Problem 7G.2 of \cite{GJ}.) In particular, $C_0(X)=C_{00}(X)$ implies that $X$ is countably compact. Every countably compact paracompact space is necessarily compact. (See Theorem 5.1.20 of \cite{E}.) Therefore, if $X$ is non-compact, then $C_0(X)=C_{00}(X)$ implies that $X$ is non-paracompact (and thus, in particular, non-metrizable and non-Lindel\"{o}f). (Every Lindel\"{o}f space, in particular every compact space, and every metrizable space is paracompact; see Theorems 5.1.1--5.1.3 of \cite{E}.)

\begin{theorem}\label{KGDD}
Let $X$ be a normal space. Let $H$ be a non-vanishing self-adjoint closed subalgebra of $C_B(X)$ which has local units. The following are equivalent:
\begin{itemize}
\item[(1)] For any $f$ in $H$ there is some $h$ in $H$ with $\mathrm{supp}(f)\subseteq\mathrm{int}_X\mathrm{supp}(h)$.
\item[(2)] Every $\sigma$-compact subspace of $\mathfrak{sp}(H)$ is contained in a compact subspace of $\mathfrak{sp}(H)$; in particular, $\mathfrak{sp}(H)$ is countably compact.
\item[(3)] $C_{00}(\mathfrak{sp}(H))=C_0(\mathfrak{sp}(H))$.
\end{itemize}
\end{theorem}

\begin{proof}
Conditions (2) and (3) are equivalent by the observation made preceding the statement of the theorem. (Note that $\mathfrak{sp}(H)=\lambda_HX$ by Theorem \ref{JUFG}.)

Let the mapping $\psi:H\rightarrow C_0(\lambda_HX)$ be defined by $\psi(f)=f_H$ for any $f$ in $H$. Then $\psi$ is an isometric isomorphism by (the proof of) Theorem \ref{FS}. Note that $(1)$ is equivalent to $H=\psi^{-1}(C_{00}(\lambda_HX))$ by Lemma \ref{FYS}. But this holds if and only if $\psi(H)=C_{00}(\lambda_HX)$ (and since $\psi$ is surjective) if and only if $C_0(\lambda_HX)=C_{00}(\lambda_HX)$. This is equivalent to (2).
\end{proof}

In the above theorem we considered compactness properties (such as countable compactness) of the spectrum. In the next few theorems we consider connectedness properties of the spectrum. (Compare Theorems \ref{KGDD} and \ref{JJGD} with Theorems 3.2.10 and 3.2.11 of \cite{Kou6}, respectively.)

We first consider the case of usual connectedness. We will use the well known fact that the Stone--\v{C}ech compactification $\beta X$ of a completely regular space $X$ is connected if and only if $X$ is connected. (See Problem 6L of \cite{GJ}.) Recall that an algebra is said to be \textit{indecomposable} if it has no idempotent except $\mathbf{0}$ (and $\mathbf{1}$, if it is unital).‎

\begin{theorem}\label{JJGD}
Let $X$ be a completely regular space. Let $H$ be a non-vanishing self-adjoint closed subalgebra of $C_B(X)$ which has local units. The following are equivalent:
\begin{itemize}
\item[(1)] ‎‎$X‎$ ‎is connected.‎
\item[(2)] ‎‎$\mathfrak{sp}(H)‎$ ‎is connected.‎
\item[(3)] $‎H‎$ ‎is ‎indecomposable.‎
\end{itemize}
\end{theorem}

\begin{proof}‎
Note that $‎\lambda_HX‎$ (‎$=\mathfrak{sp}(H)‎$ by Theorem \ref{JUFG}) contains $X$ by Lemma \ref{TFRS}, and is contained in $\beta X$. Therefore $\beta(\lambda_HX)=\beta X$.‎ Thus, ‎‎$X‎$ ‎is connected if and only if $\beta X$ is connected if and only if $\beta(\lambda_HX)$ is connected if and only if $\lambda_HX$ is connected. This shows the equivalence of (1) and (2).

(1) \textit{implies} (3). Suppose that $‎H‎$ ‎is not ‎indecomposable.‎ Then $H$ has an idempotent $f$ other than $\mathbf{0}$ and $\mathbf{1}$. Clearly, $f$ is $\{0,1\}$-valued. In particular, $f^{-1}(0)$ and $f^{-1}(1)$ form a separation for $X$, and $X$ is therefore disconnected.

(3) \textit{implies} (1). Suppose that ‎‎$X‎$ ‎is disconnected. Let $U$ and $V$ be a separation for ‎‎$X‎$. Then $U$ and $V$ are disjoint closed subspaces of $X$ and therefore by our assumption ‎‎there ‎is some ‎‎$‎f‎$ ‎in ‎‎$‎H‎$ ‎such ‎that ‎‎$‎f\mid_U=\mathbf{0}‎$ ‎and ‎‎$‎f\mid_V=\mathbf{1}‎$‎. That is, $f$ is the characteristic function $\chi_V$ on $X$. Clearly, $f$ is idempotent and $f$ is neither $\mathbf{0}$ nor $\mathbf{1}$. That is $H$ is not ‎indecomposable.‎
‎\end{proof}

We do not know how Theorem \ref{JJGD} can be formulated in the context of local connectedness. For possible future reference we record this below as an open question ‎

\begin{question}\label{GHFA}
Let $X$ be a completely regular space. For a non-vanishing self-adjoint closed subalgebra $H$ of $C_B(X)$ find necessary and sufficient conditions $($in terms of either $X$ or $H$$)$ for ‎‎$\mathfrak{sp}(H)‎$ to be locally connected.‎
\end{question}

As we have just pointed out, we do not know the local connectedness version of Theorem \ref{JJGD}; however, we can obtain some results in the presence of a compactness property called \textit{pseudocompactness}. Recall that a space $X$ is called \textit{pseudocompact}‎ if there is no unbounded continuous scalar valued mapping on $X$.

In the next theorem we need to use the following lemma which is due to Henriksen and Isbell in \cite{HI} (that (1) implies (2) is actually due to Banaschewski in \cite{B}).

\begin{lemma}\label{FPO}
Let $X$ be a completely regular space. The following are equivalent:
\begin{itemize}
\item[(1)] ‎‎$X‎$ ‎is locally connected and pseudocompact‎.‎
\item[(2)] ‎‎$\beta X$ is locally connected.
\end{itemize}
\end{lemma}

\begin{theorem}\label{JHGO}
Let $X$ be a completely regular space. Let $H$ be a non-vanishing self-adjoint closed subalgebra of $C_B(X)$ which has local units. The following are equivalent:
\begin{itemize}
\item[(1)] ‎‎$\mathfrak{sp}(H)‎$ is locally connected and pseudocompact‎.‎
\item[(2)] ‎‎$X‎$ ‎is locally connected and pseudocompact‎.‎
\end{itemize}
\end{theorem}

\begin{proof}‎
Note that $\beta‎(‎\lambda_HX‎)=\beta ‎X$, as $X\subseteq‎\lambda_HX\subseteq\beta ‎X$ by Lemma \ref{TFRS}. It now follows from Lemma \ref{FPO} that $X‎$ ‎is locally connected and pseudocompact‎ if and only if $\beta ‎X=\beta‎(‎\lambda_HX)$ is locally connected if and only if $‎\lambda_HX‎$ (‎$=\mathfrak{sp}(H)‎$ by Theorem \ref{JUFG}) ‎is locally connected and pseudocompact.
\end{proof}

In the next few results we study various (dis)connectedness properties of the spectrum. These properties are total disconnectedness, zero-dimensionality, strong ‎zero-dimensionality, basic disconnectedness, extreme disconnectedness, and being an $F$-space. There is some disagreement on definition of some of these properties, so we define them below to avoid confusion. Recall that two subspaces $‎A$ and $B‎$ ‎of a space $X$ are called \textit{completely separated} if there is a continuous mapping $f:X\rightarrow[0,1]$ such that $f\mid_A=\mathbf{1}$ and $f\mid_B=\mathbf{0}$.

Let $X$ be a completely regular space. The space $X$ is called
\begin{enumerate}
  \item \textit{totally disconnected} if $X$ does not contain any connected subspace of cardinality larger than one, or, equivalently, if every component of $X$ is a singleton.
  \item \textit{‎zero-dimensional} if the set of all open and closed subspaces of $X$ forms an open base for $X$.
  \item \textit{strongly ‎zero-dimensional} if every two completely separated subspaces of $X$ are separated by two disjoint ‎open and closed ‎subspaces.
  \item \textit{basically disconnected} if the closure of every cozero-set of $X$ is open.
  \item \textit{extremally disconnected} if the closure of every open subspace of $X$ is open.
  \item \textit{an $F$-space} if any two disjoint cozero-sets in $X$ are completely separated.
\end{enumerate}

See Section 6.2 of \cite{E} for definitions of (1)--(3) and (5), and see Problems 1H and 14N of \cite{GJ} for definitions of (4) and (6), respectively. It is known that $(5)\Rightarrow(4)\Rightarrow(3)\Rightarrow(2)\Rightarrow(1)$ and $(4)\Rightarrow(6)$. (See Theorem 6.2.1 of \cite{E} for $(2)\Rightarrow(1)$, and Theorem 6.2.6 of \cite{E} for $(3)\Rightarrow(2)$. To see $(4)\Rightarrow(3)$, let $X$ be a basically disconnected space. Then $\beta X$ is basically disconnected by Problem 6M.1 of \cite{GJ}. But every basically disconnected space is zero-dimensional by Problem 4K.8 of \cite{GJ}. Therefore $\beta X$ is zero-dimensional, and thus is  strongly zero-dimensional by Theorem 6.2.7 of \cite{E}, as it is compact. But $\beta X$ is strongly zero-dimensional if and only if $X$ is strongly zero-dimensional by Theorem 6.2.12 of \cite{E}. The implication $(5)\Rightarrow(4)$ is clear. The implication $(4)\Rightarrow(6)$ follows from Problem 14N.4 of \cite{GJ}.)

\begin{lemma}\label{PFD}
Let $X$ be a completely regular space. Let $G$ and $H$ be non-vanishing self-adjoint closed subalgebras of $C_B(X)$ which have local units. Then $G=H$ if $‎\lambda_GX=‎\lambda_HX$.
\end{lemma}

\begin{proof}
Suppose that $‎\lambda_GX=‎\lambda_HX$. Let
\[\phi:G\longrightarrow C_0(\lambda_GX)\quad\text{and}\quad\psi:H\longrightarrow C_0(\lambda_HX)\]
where $\phi(g)=g_G$ and $\psi(h)=h_H$ for any $g$ in $G$ and $h$ in $H$. Then $\phi$ and $\psi$ are isomorphisms by (the proof of) Theorem \ref{FS}, where
\[\phi^{-1}(g)=g\mid_X\quad\text{and}\quad\psi^{-1}(h)=h\mid_X\]
for any $g$ in $C_0(\lambda_GX)$ and $h$ in $C_0(\lambda_HX)$. The mapping
\[\psi^{-1}\phi:G\longrightarrow C_0(\lambda_GX)=C_0(\lambda_HX)\longrightarrow H\]
is such that
\[\psi^{-1}\phi(g)=\psi^{-1}(g_G)=g_G\mid_X=g\]
for any $g$ in $G$. Thus, in particular $G\subseteq H$. Similarly $H\subseteq G$. Therefore $G=H$.
\end{proof}

Let $\mathscr{P}$ be either strong ‎zero-dimensionality, basic disconnectedness, extreme disconnectedness, or being an $F$-space. Let $X$ be a completely regular space. It is known that $X$ has $\mathscr{P}$ if and only if $\beta X$ has $\mathscr{P}$. (See Theorem 6.2.12 of \cite{E} for strong ‎zero-dimensionality, Problem 6M.1 of \cite{GJ} for basic disconnectedness and extreme disconnectedness and Theorem 14.25 of \cite{GJ} for being an $F$-space.)

\begin{theorem}\label{PFS}
Let $X$ be a completely regular space. Let $H$ be a non-vanishing self-adjoint closed subalgebra of $C_B(X)$ which has local units. Let $\mathscr{P}$ be either strong ‎zero-dimensionality, basic disconnectedness, extreme disconnectedness, or being an $F$-space. The following are equivalent:
\begin{itemize}
\item[(1)] ‎‎‎‎‎$\mathfrak{sp}(H)‎$ has $\mathscr{P}$.
\item[(2)] $‎X‎$ has $\mathscr{P}$.
\end{itemize}
\end{theorem}

\begin{proof}
Note that $\lambda_HX$ has $\mathscr{P}$ if and only if $\beta(\lambda_HX)‎$ has $\mathscr{P}$. But $\beta(\lambda_HX)‎=\beta X$, as $X‎\subseteq‎\lambda_HX‎‎\subseteq‎\beta ‎X$ by Lemma \ref{TFRS}. Therefore $\mathfrak{sp}(H)‎‎$ ($=\lambda_HX$ by Theorem \ref{JUFG}) has $\mathscr{P}$ if and only if $\beta ‎X$ has $\mathscr{P}$ if and only if $‎X$ has $\mathscr{P}$.
\end{proof}

\begin{theorem}\label{POFS}
Let $X$ be a completely regular space. Let $H$ be a non-vanishing self-adjoint closed subalgebra of $C_B(X)$ which has local units. Let ‎‎‎‎‎$\mathfrak{sp}(H)‎$ be either strongly ‎zero-dimensional, basically disconnected, or extremally disconnected. Then
\[H=\overline{\langle\, f\in H:\mbox{$f$ is an idempotent}\,\rangle}.\]
Here the bar denotes the closure in $C_B(X)$.
\end{theorem}

\begin{proof}
Note that $\lambda_HX$ ($=\mathfrak{sp}(H)‎$ by Theorem \ref{JUFG}) is strongly ‎zero-dimensional, as basic disconnectedness and extreme disconnectedness are stronger than strong ‎zero-dimensionality. Let $G=\overline{I}$ where $I$ is the subalgebra of $H$ generated by the set of all its idempotents. We verify that $G$ is a non-vanishing self-adjoint closed subalgebra of $C_B(X)$ which has local units such that ‎$\lambda_GX=‎\lambda_HX$. Corollary \ref{PFD} will then imply that $G=H$.

To show that $G$ is non-vanishing, let $x$ be in $X$. Then $x$ is in $\lambda_HX$. Since $\beta X$ is  zero-dimensional, there is an open and closed neighborhood $U$ of $x$ in $\beta X$ such that $U\subseteq\lambda_HX$. By Lemma \ref{ID} then $U\subseteq\mathrm{cl}_{\beta X}\mathrm{coz}(h)$ for some $h$ in $H$. In particular $X\cap U\subseteq\mathrm{cl}_X\mathrm{coz}(h)=\mathrm{supp}(h)$. Note that $X\cap U$ is open and closed in $X$. Therefore, by our assumption, there is an $f$ in $H$ such that $f\mid_{X\cap U}=\mathbf{1}$ and $f\mid_{X\setminus U}=\mathbf{0}$. Observe that $f$ is in $G$, as it is idempotent, and $f(x)\neq 0$.

We now show that $G$ is self-adjoint. Note that elements of $I$ are sums of the form
\[\sum_{k=1}^n\alpha_kf^k_1\cdots f^k_{j_k}\]
where $\alpha_j$'s are scalars and $f^i_j$'s are idempotent elements of $H$. It is now clear that for any elements $f$ in $I$ its conjugate $\overline{f}$ is also in $I$. Let $g$ be in $G$. Then $f_n\rightarrow g$ for some sequence $f_1,f_2,\ldots$ in $I$. But $\overline{f_n}\rightarrow\overline{g}$ and $\overline{f_1},\overline{f_2},\ldots$ are in $I$. Therefore $\overline{g}$ is in $G$.

Next, we show that $G$ has local units. Suppose that $A$ is a closed subspaces of $X$ and $U$ is a neighborhood of $A$ in $X$ which is contained in $\mathrm{supp}(g)$ for some $g$ in $G$. Note that $g$ is also in $H$. Thus, by our assumption, there is some $f$ in $H$ such that $f\mid_A=\mathbf{1}$ and $f\mid_{X\setminus U}=\mathbf{0}$. Now, since $X$ is   strongly zero-dimensional, there is an open and closed subspace $V$ of $X$ such that $A\subseteq V\subseteq U$. Again, $V$ is a closed subspace of $X$ which is also a neighborhood of itself in $X$ and is contained in $\mathrm{supp}(g)$. Therefore, by our assumption, there is some $h$ in $H$ such that $h\mid_V=\mathbf{1}$ and $h\mid_{X\setminus V}=\mathbf{0}$. Observe that $h$ is an idempotent in $H$ and is therefore in $G$. Also, $h\mid_A=\mathbf{1}$ and $h\mid_{X\setminus U}=\mathbf{0}$, as $A\subseteq V$ and $X\setminus U\subseteq X\setminus V$.

Finally, we show that ‎$\lambda_GX=‎\lambda_HX$. It is clear that $\lambda_GX\subseteq‎\lambda_HX$, as $G\subseteq‎ H$. To check the reverse inclusion, let $t$ be in $‎\lambda_HX$. By an argument similar to the one we used to check that $G$ is non-vanishing we can find an open and closed neighborhood $U$ of $t$ in $\beta X$ such that $f=\chi_{(X\cap U)}$ is in $H$. Observe that $\mathrm{cl}_{\beta X}(X\cap U)=\mathrm{cl}_{\beta X}U=U$ is open and closed in $\beta X$ and
\[U=\mathrm{int}_{\beta X}\mathrm{cl}_{\beta X}U=\mathrm{int}_{\beta X}\mathrm{cl}_{\beta X}(X\cap U)=\mathrm{int}_{\beta X}\mathrm{cl}_{\beta X}\mathrm{coz}(f)\subseteq\lambda_GX.\]
Thus $t$ is in $\lambda_GX$. Therefore $\lambda_HX\subseteq‎\lambda_GX$.
\end{proof}

It is known that in the class of locally compact paracompact spaces, total disconnectedness, zero-dimensionality and strong zero-dimensionality are all equivalent. (See Theorem 6.2.10 of \cite{E}.)

\begin{theorem}\label{PO}
Let $X$ be a completely regular space. Let $H$ be a non-vanishing self-adjoint closed subalgebra of $C_B(X)$ which has local units. Let ‎‎‎‎‎$\mathfrak{sp}(H)‎$ be paracompact $($in particular, metrizable or Lindel\"{o}f\,$)$. The following are equivalent:
\begin{itemize}
\item[(1)] ‎‎‎‎‎$\mathfrak{sp}(H)‎$ is totally disconnected.
\item[(2)] $\mathfrak{sp}(H)‎$ is zero-dimensional.
\item[(3)] $\mathfrak{sp}(H)‎$ is strongly zero-dimensional.
\item[(4)] $X‎$ is strongly zero-dimensional.
\end{itemize}
\end{theorem}

\begin{proof}
Note that ‎‎‎‎‎$\mathfrak{sp}(H)$ ($=‎\lambda_HX$ by Theorem \ref{JUFG}) is locally compact (by its definition). Thus (1)--(3) are equivalent if $\mathfrak{sp}(H)‎$ is also paracompact. The equivalence of (3) and (4) follows from Theorem \ref{PFS}.
\end{proof}

\section*{Acknowledgements}

The authors also thank the anonymous reviewer for careful reading of the manuscript, useful comments, and prompt report.

\end{document}